\documentclass{mpi2015-cscpreprint}

\usepackage[american]{babel}
\usepackage{amssymb}
\usepackage{amsthm,amsmath}
\usepackage{algorithm}%
\usepackage{algorithmicx}%
\usepackage{algpseudocode}%
\usepackage{graphicx}
\usepackage{appendix}
\usepackage{multirow}
\DeclareMathOperator*{\esssup}{ess\,sup}

\newtheorem{theorem}{Theorem}[section]

\newtheorem{lemma}[theorem]{Lemma}

\newtheorem{remark}{Remark}[section]

\numberwithin{equation}{section}

\begin{document}

\title{Iterative approximations of periodic trajectories for nonlinear systems with discontinuous inputs}

\author{Alexander~Zuyev$^{1,2}$, Peter~Benner}
\affil[1]{Max Planck Institute for Dynamics of Complex Technical Systems, Magdeburg, Germany}
\affil[2]{Institute of Applied Mathematics and Mechanics, National Academy of Sciences of Ukraine}




\abstract{
Nonlinear control-affine systems described by ordinary differential equations with bounded measurable input functions are considered.
The {
solvability of general boundary value problems} for these systems is formulated in the sense of Carathéodory solutions.
It is shown that, under the dominant linearization assumption, the {considered class of boundary value problems} admits a unique solution for any admissible control. {These solutions} can be obtained as the limit of the proposed simple iterative scheme and, {in the case of periodic boundary conditions, via the developed Newton-type schemes}.
Under additional technical assumptions, sufficient contraction conditions of the corresponding generating operators are derived analytically.
The proposed iterative approach is applied to compute periodic solutions of a realistic chemical reaction model with discontinuous control inputs.
}

 \novelty{
\begin{itemize}
\item Existence and uniqueness of periodic solutions for nonlinear control systems with general measurable input functions
\item Iterative schemes and numerical implementation of an algorithm for approximating the periodic solutions of discontinuous systems
\item Analytical sufficient conditions for the convergence of the simple iteration and Newton-type methods
\item Approximate periodic trajectories of a controlled chemical reaction model under arbitrary switching strategies
\end{itemize}}

\maketitle

\section{Introduction}
\vskip-1ex
Periodic optimal control problems have been attracting considerable interest in the mathematical literature~\cite{gilbert1977optimal,colonius2006optimal,bayen2019optimal,guilmeau2022singular} and play a significant role in a variety of emerging engineering applications (see, e.g.,~\cite{pittelkau1993optimal,varigonda2008optimal,peng2011optimal,bernard2023optimal,elgindy2023new} and references therein).
Our current paper is motivated by the previous analysis of nonlinear optimization problems with isoperimetric constraints~\cite{zuyev2017isoperimetric,benner2019periodic}, where the main goal is to optimize the cost functional on periodic trajectories under discontinuous control strategies.
A crucial ingredient for achieving this goal is the characterization of the set of periodic solutions for a nonlinear control system with bang-bang inputs.

In the paper~\cite{benner2019periodic}, the Chen--Fliess series have been exploited for constructing the $\tau$-periodic solutions of nonlinear control-affine systems with switching.
This approach allows representing the initial data of the considered system as a solution of nonlinear algebraic equations, whose order depends on the remainder of the Chen--Fliess series for small $\tau>0$.
{Note that the solutions to control-affine systems can be approximated using Lie-algebraic techniques developed for parameterized systems and systems with jumps, as described in~\cite{treanctua2013weak,treanctua2017local}.}
The construction of~\cite{benner2019periodic} has also been extended to a class of nonlinear chemical reaction models in non-affine form in~\cite{benner2021analysis}.
{It should be emphasized} that the convergence of the Chen--Fliess expansion is not guaranteed over an arbitrary time interval $[0,\tau]$,
even for systems with analytic vector fields.
{Therefore, alternative approximation techniques should be developed to design periodic solutions for nonlinear control systems with large periods.}
We note that the results of~\cite{zuyev2017isoperimetric,benner2019periodic,benner2021analysis} have been applied to nonlinear chemical reaction models in the case of small periods of the control functions.
In the recent paper~\cite{JOTA2024}, it was shown that the optimal controls are bang-bang for an isoperimetric optimization problem described by a nonlinear hyperbolic partial differential equation with periodic boundary conditions. The optimization of chemical reactions using periodic input modulations is of great importance in modern chemical engineering~\cite{felischak2021}, as it supports to develop technology that increases the yield of reaction products as compared to steady operation.
For different classes of nonlinear ordinary differential equations with regular right-hand sides,
the problems of existence and approximation of periodic solutions have been studied by
the method of generalized quasilinearization~\cite{lakshmikantham1996methods},
comparison techniques~\cite{vseda1992periodic},
fixed point theory~\cite{yao2005existence},
reproducing kernel methods~\cite{al2016numerical}, and other techniques.
The presented list of references does not pretend to be exhaustive.

To the best of our knowledge, a complete characterization of periodic trajectories for a given nonlinear control system with discontinuous inputs remains an open problem.
The purpose of this paper is to provide such a characterization, along with efficient computational methods for approximating solutions to the boundary value problem:
\begin{eqnarray}
&\dot x(t) = A x(t) + g(x(t)) + u(t),\;\; x\in D\subset {\mathbb R}^n,\; t\in [0,\tau],
\label{sys_switchings}\\
\label{BC}
&x(\tau)=x(0),
\end{eqnarray}
where $x=(x_1,...,x_n)^\top$, $D$ is a domain containing the origin $0\in\mathbb R^n$, $A$ is a constant $n\times n$ matrix,
$g(\cdot)\in C^1\left (D; {\mathbb R}^n \right)$ represents the nonlinearity, and $u(\cdot)\in L^\infty \left( [0,\tau]; {\mathbb R}^n \right)$ is an arbitrary control function.

As the function $u(t)$ is allowed to be discontinuous, we treat the solutions of~\eqref{sys_switchings} in the sense of Carath\'eodory~\cite[Ch.~1]{filippov2013differential} as solutions of the integral equation
$$
{x(t) = x_0 + \int_0^t \left[ A x(s)+ g(x(s)) + u(s)  \right] \,ds,\; x_0= x(0).}
$$
The existence and uniqueness results for solutions to differential equations, with the right-hand sides being continuous in $x$ and discontinuous in $t$ under the Carathéodory conditions, are summarized in~\cite[Ch.~1]{filippov2013differential}.
By using the variation of constants method and introducing the matrix exponential $e^{tA}$, we can rewrite the above integral equation in the form
\begin{equation}
x(t) = e^{tA} x_0 + \int_0^t e^{(t-s)A} \left[g(x(s))+u(s) \right] \,ds.
 \label{intexpA}
\end{equation}
Let $x(t)$ be a solution of~\eqref{intexpA} with some $u(t)$ on $[0,\tau]$ such that $x(\tau)=x(0)$,
then $x(t)$ can be extended to the $\tau$-periodic function $\tilde x(t)$, defined for all $t\in \mathbb R$.
If, moreover, the function $\tilde u(t)$ is $\tau$-periodic and $\tilde u(t)=u(t)$ for all $t\in[0,\tau)$,
then $\tilde x(t)$ satisfies~\eqref{intexpA} with the input $\tilde u(t)$ for all $t\in \mathbb R$.
Because of this simple fact, we will refer to the Carath\'eodory solutions of boundary value problem~\eqref{sys_switchings}--\eqref{BC} as periodic solutions.
If $x(t)$ is a solution of~\eqref{sys_switchings}--\eqref{BC}, then formula~\eqref{intexpA} allows to represent its initial value $x_0=x(0)=x(\tau)$:
\begin{equation}
x_0 = (e^{-\tau A}-I)^{-1} \int_0^\tau e^{-sA} \left[g(x(s))+u(s) \right] \,ds,
 \label{x0}
\end{equation}
provided that
$$
\det \left(e^{-\tau A}-I\right)\neq 0,\eqno(A1)
$$
where $I$ is the identity matrix.
An immediate consequence of equation~\eqref{x0} {\em for the case of linear systems on $D=\mathbb R^n$} is that the system~\eqref{sys_switchings} with $g(x)\equiv 0$ and any $u(\cdot)\in  L^\infty \left( [0,\tau]; {\mathbb R}^n \right)$
{\em has a unique periodic solution} $x(t)$ on $t\in [0,\tau]$, and its initial data $x_0=x(0)$ is defined by~\eqref{x0} if assumption~$(A1)$ holds.
The existence of periodic solutions has been studied {\em for weakly nonlinear} boundary value problems with piecewise-constant right-hand sides in~\cite{benner2023periodic}. That paper develops perturbation analysis techniques for nonlinear differential equations with switching under periodic boundary conditions.
A modified iterative scheme is proposed there for constructing approximate periodic solutions. Note that the contribution of~\cite{benner2023periodic} is limited to systems of ordinary differential equations with a small parameter, and future studies
are needed for systems with general nonlinearities and merely measurable controls.

The rest of this paper is organized as follows.
In Section~\ref{iterative}, a simple iteration method is adapted to construct solutions of~\eqref{sys_switchings} under
{general affine boundary conditions and} fairly mild assumptions on $u(t)$ and $g(x)$. The convergence of this scheme is justified in terms of the growth rate of the matrix exponential $e^{tA}$ and the Lipschitz constant of $g(x)$ (Theorem~\ref{convtheorem}).
{This novel contribution encompasses the existence of solutions not only under two-point boundary conditions (including periodic ones), but also under abstract boundary conditions.}
The operator formulation of the periodic boundary value problem{~\eqref{sys_switchings}--\eqref{BC}} is exploited in Section~\ref{newton} to derive Newton's method.
Sufficient convergence conditions for Newton's iterations are explicitly stated in Theorem~\ref{thm_NK}, with the proof provided in the Appendix.
The proposed iterative approach is applied to a nonlinear chemical reaction model in Section~\ref{application}. The final conclusions and perspectives are outlined in Section~\ref{conc}.

{\em Notations}.
For further analysis, let $X=C\left ( [0,\tau]; \mathbb R^n \right)$ denote the Banach space of all continuous functions $x(t)$ from $[0,\tau]$ to $\mathbb R^n$, equipped with the norm
$
\|x(\cdot)\|_X := \sup_{t \in [0,\tau]} \|x(t)\|
$,
where $\|\xi\|$ is the Euclidean norm of a column vector $\xi\in \mathbb R^n$.
{We also introduce the subspace $\mathring X = \{x\in X\,|\, x(0)=0\}$ of $X$
and, for brevity, we will}
omit the subscript $X$ referring to the norm of elements of $X$.
The Euclidean norm induces the $2$-norm $\|A\|:=\sup_{\|\xi\|\le 1} \|A\xi\|$ of a matrix $A$, which will be used throughout the text.
The space of bounded linear operators, acting from $X$ {to a normed vector space $Y$}, is denoted by ${\mathcal L}(X,{Y})$.
{
We will refer to an operator ${\cal B}\in {\mathcal L}(X,{\mathbb R}^n)$ as a bounded linear vector functional.
The value of a vector functional ${\cal B}:X\to {\mathbb R}^n$ at $x\in X$ is denoted by ${\cal B}(x)$,
and by ${\cal B}(x(t))$ when $x(t)$ is explicitly defined within the corresponding expression.
}
We will also use the notations $X_D=C\left ( [0,\tau]; D \right)$ and $X_{D'}=C\left ( [0,\tau]; D' \right)$ when $D'\subset D$ is a closed domain. The Jacobian matrix of a function $g\in C^1(D; {\mathbb R}^n)$ will be denoted as $g'(\xi)=\frac{\partial g(\xi)}{\partial
\xi}$. For a function $x\in X$, $g'(x(s))$ stands for the substitution of $\xi=x(s)$ in $g'(\xi)$.

\section{Simple iteration method}\label{iterative}
{
We consider boundary condition~\eqref{BC} as a particular case of general affine conditions:
\begin{equation}\label{affine_BC}
{\cal B} (x) = \beta,
\end{equation}
where ${\cal B}:X\to {\mathbb R}^n$ is a bounded linear functional and $\beta \in {\mathbb R}^n$ is a given vector.
It is easy to see that~\eqref{BC} can be written in the form~\eqref{affine_BC} with
\begin{equation}\label{BC_periodic_abstract}
{\cal B} (x) = x(\tau)-x(0),\; \beta = 0.
\end{equation}
For a vector function $\xi(t)=e^{tA} \xi_0$ with $\xi_0\in\mathbb R^n$, the value of ${\cal B} (\xi)$ is represented as ${\cal B} (\xi)={\cal B}_\tau \xi_0$, under the convention that ${\cal B}_\tau = {\cal B}(e^{tA})$ denotes the corresponding constant $n\times n$-matrix.
Assuming that ${\cal B}_\tau$ is nonsingular,}
consider the following nonlinear operator ${\cal F}: X_D \to X $:
\begin{equation}
{\cal F} : x \mapsto {\cal F}(x)(t) = e^{tA} c(x) + \int_0^t e^{(t-s)A} \left[g(x(s))+u(s) \right] \,ds,
\label{F_op}
\end{equation}
where the vector functional $c:X_D\to {\mathbb R}^n$ is defined by
\begin{equation}
c(x) = {{\cal B}_\tau^{-1} \beta - {\cal B}_\tau^{-1} {\cal B} \left( \int_0^t e^{(t-s)A} \left[g(x(s))+u(s)\right] \,ds\right)}.
\label{c_fun}
\end{equation}
{We observe that the argument of $\cal B$ vanishes at $t=0$ in~\eqref{c_fun} and,
for further analysis, we define the norm of the restriction of $\cal B$ to $\mathring X$ as
$\|{\cal B}\|_0:=\sup_{x\in \mathring X,\|x\|\le 1} \|{\cal B}(x)\|$.
}
%
If an initial function $x^{(0)}\in  X_D$ is given, we generate the sequence of functions $x^{(k)}={\cal F}(x^{(k-1)})$ for $k=1,2,...$ .
This sequence is well-defined, particularly when $X_D=X$ (i.e., $D=\mathbb R^n$). If $D\neq \mathbb R^n$, we consequently assume that  ${\cal F}(x^{(k)})(t)\in D$ for all $t\in [0,\tau]$, $k=1,2,... $ .
Below, we propose sufficient conditions for the convergence of $x^{(k)}(t)$.

\begin{theorem}\label{convtheorem}{\em
Assume that {$\det ({\cal B}_\tau)\neq 0$,} and that there exist constants $L\ge 0$, $M\ge 1$, $\omega>0$,
along with a closed convex domain $D'\subset D$, such that:
$$
\left\| g'(\xi) \right\|\le L,\; \|e^{tA}\| \le M e^{\omega \vert t \vert }\; \text{for all}\; \xi\in D,\; t\in [-\tau,\tau],
\eqno(A2)
$$
\vskip-2ex
$$
{\cal F} (x) \in X_{D'}\;\; \text{for each}\;\; x \in X_{D'},
\eqno(A3)
$$
\vskip-2ex
$$
{\frac{L M (e^{\omega\tau}-1)\left( 1+ M \| {\cal B}_\tau^{-1} \| e^{\omega\tau}  \|{\cal B}\|_0 \right)}{\omega}  < 1.}
\eqno(A4)
$$
\vskip-1ex
Then, for any $x^{(0)}\in X_{D'} $, the sequence
$
x^{(k)} = {\cal F}(x^{(k-1)})$, $k=1,2,...$,
converges to the limit $x^*\in X_{D'}$ as $k\to \infty$. This limit function $x^*(t)$ is the unique solution of {the boundary value problem~\eqref{sys_switchings}, \eqref{affine_BC}}.}
\end{theorem}
\begin{proof}
Under {the assumptions that ${\cal B}_\tau$ is nonsingular and $(A2)$ holds}, the operator ${\cal F}:X_D\to X$ defined by~\eqref{F_op}--\eqref{c_fun} is Fréchet differentiable.
Indeed, the Fréchet derivative ${\cal F}'(x):X\to X$ at $x \in X_D$ is a bounded linear operator, and its action on $\delta x \in X$ is defined as follows:
\begin{equation}
{\cal F}'(x) (\delta x)(t) = e^{tA} c'_x (\delta x)
+\int_0^t e^{(t-s)A} g'(x(s)) \delta x(s) \, ds,
\label{dF}
\end{equation}
where
\begin{equation}
c'_x (\delta x) ={ - {\cal B}_\tau^{-1} {\cal B}\left(\int_0^t e^{(t-s)A} g'(x(s)) \delta x(s) \, ds\right)}.
\label{dc}
\end{equation}
By taking into account inequalities~$(A2)$, we conclude that
{
$$\small
\begin{aligned}
&\|{\cal F}'(x) (\delta x)\|  \le   Me^{\omega \tau} \|c'_x(\delta x)\| + \sup_{t\in [0,\tau]} \int_0^t \|e^{(t-s)A} g'(x(s))\|ds \|\delta x\| \\
&
\le \frac{L M (e^{\omega\tau}-1)}{\omega} \left( 1+ M \| {\cal B}_\tau^{-1} \| e^{\omega\tau}  \|{\cal B}\|_0 \right)  \|\delta x\|.
\end{aligned}
$$
}
We see that the operator $\cal F$ is contractive on $X_{D'}$ if condition~$(A4)$ holds. Additionally, ${\cal F}(X_{D'})\subseteq X_{D'}$ due to assumption~$(A3)$,
and the metric space $\left(X_{D'},d\right)$ equipped with the distance $d(x,y)=\|x-y\|_X$ is complete.
 Thus,  the assertion of Theorem~\ref{convtheorem} follows from the Banach fixed point theorem~\cite[Ch.~XVI, \S 1, Thm.~1 and Ch.~XVII, \S 1, Thm.~1]{kantorovich1982functional}.
\end{proof}

\begin{remark}
The initial function $x^{(0)}(t)$ can be taken, in particular, as $x^{(0)}(t)\equiv 0$. In this case, it is easy to see that the first approximation $x^{(1)}(t)$ is the solution of the system $\dot x^{(1)} = Ax^{(1)} + u(t) +g(0)$, $t\in [0,\tau]$, such that {${\cal B}(x^{(1)}) = \beta$}.
\end{remark}

{
\begin{remark}
To apply Theorem~\ref{convtheorem} to the periodic  problem~\eqref{sys_switchings}--\eqref{BC}, we express the boundary conditions~\eqref{affine_BC} using~\eqref{BC_periodic_abstract} and note that
${\cal B}_\tau = e^{\tau A} - I$, so the nonsinularity of ${\cal B}_\tau$ is equivalent to~$(A1)$.
In this case, $\|{\cal B}\|_0 = 1$, and the vector functional~\eqref{c_fun} and its derivative take the form:
\begin{equation}
c(x) = (e^{-\tau A}-I)^{-1} \int_0^\tau e^{-sA} \left[g(x(s))+u(s)\right] \,ds,
\label{c_fun_old}
\end{equation}
\begin{equation}
c'_x (\delta x) = (e^{-\tau A}-I)^{-1}\int_0^\tau e^{-sA} g'(x(s)) \delta x(s) \, ds.
\label{dc_old}
\end{equation}
\end{remark}}

\section{Newton's method}\label{newton}

In this section, we develop Newton's method for finding the fixed points of ${\cal F}$ {in the case of periodic boundary conditions~\eqref{affine_BC},~\eqref{BC_periodic_abstract}}. For this purpose, we introduce the nonlinear operator
$P:X_D \to X$ such that
\begin{equation}
P(x)  =  {\cal F}(x) - x,
\label{p_map}
\end{equation}
where ${\cal F}:X_D \to X$ is given by~\eqref{F_op}{, with the functional $c(x)$ defined in~\eqref{c_fun_old}}.
It is clear that the fixed points of $\cal F$ are solutions to the equation $P(x)=0$.
%

Given an initial function $x^{(0)}\in X_D$ such that the operator $[P'(x^{(0)})]^{-1}$ is nonsingular, we define the sequence of modified Newton's iterations $\{ x^{(k)}\}_{k=0}^\infty$ by the rule~\cite{kantorovich1982functional}:
\begin{equation}
 x^{(k+1)} =  x^{(k)} - [P'(x^{(0)})]^{-1} P( x^{(k)}), \quad k=0,1,\dots\, .
\label{newton_modified}
\end{equation}
We also consider the ``classical'' Newton's sequence $\{\tilde x^{(k)}\}_{k=0}^\infty$, defined as
\begin{equation}
\tilde x^{(k+1)} = \tilde x^{(k)} - [P'(\tilde x^{(k)})]^{-1} P(\tilde x^{(k)}),\quad \tilde x^{(0)}=x^{(0)}.
\label{newton_iterations}
\end{equation}

In order to represent the action of the operator $[P'(x)]^{-1}$ on a vector function $\delta y\in X$, we solve the following functional equation with respect to $\delta x\in X$:
$
P'(x)\delta x = \delta y
$,
or, equivalently,
$
{\cal F}'(x) (\delta x)- \delta x = \delta y
$.
The above equation can be rewritten, due to~\eqref{dF}, in the form
\begin{equation}
\small
c'_x(\delta x) + \int_0^t e^{-sA}g'(x(s)) \delta x (s) ds - e^{-tA}\delta x(t) = e^{-tA} \delta y(t).
\label{dx_eq}
\end{equation}
By differentiating this formula with respect to $t$ and introducing the function $\delta z(t)=\delta x(t) + \delta y(t)$, we obtain
\begin{equation}
\dot \delta z(t) = \left(A + g'(x(t))\right)\delta z (t) - g'(x(t)) \delta y(t).
\label{z_system}
\end{equation}
Let $\Phi_x(t)\in \textrm{Mat} (n\times n)$ be the fundamental matrix of the corresponding homogeneous system, i.e.,
\begin{equation}
\Phi_x(0)=I, \;\; \dot \Phi_x (t) = \left(A + g'(x(t)) \right) \Phi_x (t)\;\; \text{for}\; t\in [0,\tau].
\label{Phi_x}
\end{equation}
Note that the matrix $\Phi_x^{-1}(s)$ is well-defined for all $s\in [0,\tau]$ due to the uniqueness of solutions to the Cauchy problem~\eqref{Phi_x}.
Then, the variation of constants method yields the general solution of~\eqref{z_system}:
$
\delta z(t) = \Phi_x(t)\delta z(0) - \Phi_x(t)\int_0^t \Phi_x^{-1}(s) g'(x(s)) \delta y(s)ds
$.
We rewrite the above formula with respect to $\delta x(t) = \delta z(t)-\delta y(t)$ as
\begin{equation}
\small
\delta x(t) = \Phi_x(t) C_x - \Phi_x(t)\int_0^t \Phi_x^{-1}(s) g'(x(s)) \delta y(s)ds - \delta y(t).
\label{dx}
\end{equation}
The integration constant $C_x=\delta x(0)+\delta y(0)$ is defined from~\eqref{dx_eq}. Indeed, the relation~\eqref{dx_eq} at $t=0$ implies that $C_x=\delta x(0)+\delta y(0)=\delta x_c(\delta x)$. The integration constant $C$ can then be eliminated by substituting~\eqref{dx} into~{\eqref{dc_old}}. Thus,
\begin{equation}
\begin{aligned}
 C_x &= M_x^{-1} \left(e^{-\tau A}-I\right)^{-1}
\int_0^\tau e^{-tA} g'(x(t)) \\
&\times \left\{ \delta y(t) + \Phi_x(t)\int_0^t \Phi_x^{-1}(s) g'(x(s)) \delta y(s) ds \right\} dt,
\end{aligned}
\label{C_form}
\end{equation}
\begin{equation}\label{M_x}
M_x = \left(e^{-\tau A}-I\right)^{-1}\int_0^\tau e^{-tA} g'(x(t)) \Phi_x(t) dt - I.
\end{equation}
The following lemmas are needed to construct $[P'(x)]^{-1}$.

\begin{lemma}\label{lemma_X}{\em
Let assumption $(A2)$ hold, and let $x(t)\in D$ be a continuous function on $t\in [0,\tau]$. Then the matrix $\Phi_x(t)$ in~\eqref{Phi_x} satisfies the following estimates for all $t\in [0,\tau]$:
\begin{equation}\label{Phi_est}
\|\Phi_x(t)\|\le \phi_L(t), \|\Phi_x^{-1}(t)\|\le \phi_L(t), \frac{\phi_L(t)}{\sqrt{n}} =e^{(\|A\|+L)t}.
\end{equation}}
\end{lemma}
\begin{proof}
The estimates in~\eqref{Phi_est} easily follow from Grönwall's inequality.
\end{proof}

\begin{lemma}\label{lemma_Y}{\em
Let assumptions~$(A1)$--$(A2)$ be satisfied, and let
$$
S = \frac{\sqrt{n} M L R_\tau (e^{(\|A\|+L+\omega)\tau}-1)}{\|A\|+L+\omega} <1,
\eqno(A5)
$$
{
where
\begin{equation}
R_\tau = \left\|(e^{-\tau A}-I)^{-1}\right\|.
\label{RT}
\end{equation}
}
Then the matrix $M_x$ is nonsingular, and $\|M_x^{-1}\|\le \frac{1}{1-S}$.}
\end{lemma}
This lemma is proved in the Appendix.

\begin{lemma}\label{lemma_Pinv}{\em
Given $x\in X$, let the matrices $\Phi_x(t)$ and $M_x$ be defined by~\eqref{Phi_x} and~\eqref{M_x}, respectively, and assume that $M_x$ is nonsingular.
Then there exists $[P'(x)]^{-1}:X\to X$, and for any $\delta y\in X$, the value $\delta x = [P'(x)]^{-1} \delta y\in X$ is defined by formulas~\eqref{dx} and~\eqref{C_form}.
If, in addition, assumptions~$(A2)$ and $(A5)$ are satisfied, then the linear operator $[P'(x)]^{-1}$ is bounded and $\|[P'(x)]^{-1}\|\le \rho_1$, where
\begin{equation}\label{rho1}
\begin{aligned}
&\rho_1=  1+ \frac{L M R_\tau (e^{\omega\tau}-1)\phi_L(\tau)}{(1-S)\omega}\\
&+\frac{L \phi_L(\tau)(\phi_L(\tau)-\sqrt{n})}{\|A\|+L}\left(1+\frac{L M R_\tau (e^{\omega\tau}-1)\phi_L(\tau)}{(1-S)\omega}\right).
\end{aligned}
\end{equation}}
\end{lemma}
\vskip0.5ex
\begin{proof}
The above formulas~\eqref{dx}--\eqref{M_x} define the action of $[P'(x)]^{-1}$ on an arbitrary $\delta y\in X$, provided that the matrix $M_x$ is nonsingular. It remains to prove that $[P'(x)]^{-1}\in {\mathcal L}(X,X)$ under assumptions~$(A2)$ and $(A5)$. For this purpose, we estimate the norm of $\delta x\in X$ in~\eqref{dx} and the norm of $C_x\in {\mathbb R}^n$ in~\eqref{C_form} using Lemmas~\ref{lemma_X} and~\ref{lemma_Y}, along with the triangle and H\"older's inequalities:
\begin{equation}\label{deltaCest}
\small
\begin{aligned}
\|\delta x\| &\le  \phi_L(\tau) \|C_x\| + \left(\frac{L \phi_L(\tau)(\phi_L(\tau)-\sqrt{n})}{\|A\|+L} +1 \right) \|\delta y\|, \\
\|C_x\| &\le \frac{L M R_\tau (e^{\omega\tau}-1)}{(1-S)\omega}\left(1+ \frac{L \phi_L(\tau)(\phi_L(\tau)-\sqrt{n})}{\|A\|+L}\right) \|\delta y\|.
\end{aligned}
\end{equation}
Formulas~\eqref{deltaCest} imply that $\|[P'(x)]^{-1}\|\le \rho_1(x)$, with $\rho_1(x)$ expressed in~\eqref{rho1}, thereby completing the proof.
\end{proof}

To formulate sufficient conditions for the convergence of Newton's method, we assume that $g(\xi)=(g_1(\xi),...,g_n(\xi))^\top$ is twice continuously differentiable and that the Hessian matrices of its components are bounded, i.e.,
$$
g\in C^2(D;\mathbb R^n),\;
\left \| \frac{\partial^2 g_k(\xi)}{\partial \xi_i \partial \xi_j} \right\|\le \bar H\;\; \forall \xi\in D,\; k=1, ..., n,
\eqno(A6)
$$
\hskip-0.8ex
with some constant $\bar H\ge 0$. We also define the functional
$\rho_0:X\times L^\infty \left( [0,\tau]; {\mathbb R}^n \right) \to {\mathbb R}_{\ge 0}$
and the nonnegative constant $\rho_2$ by the following rules:
\begin{equation}\label{rho0}
\begin{aligned}
\rho_0(x,u) &= \|x\|+\frac{M(e^{\omega \tau}-1)(1+M R_\tau e^{\omega\tau})}{\omega}\\
&\times \left(\|u\|_{L^\infty[0,\tau]}+\sup_{t\in[0,\tau]} \|g(x(t))\|\right),
\end{aligned}
\end{equation}
\begin{equation}\label{rho2}
\rho_2 = \frac{\sqrt{n} \bar H M (e^{\omega \tau}-1)(1+M R_\tau e^{\omega\tau})}{\omega}.
\end{equation}

We summarize the convergence conditions of Newton's iteration schemes~\eqref{newton_modified} and~\eqref{newton_iterations} in the following theorem.
\begin{theorem}\label{thm_NK}{\em
Let assumptions $(A1)$, $(A2)$, $(A5)$, and $(A6)$ be satisfied.
Given $x^{(0)}\in X_D$, $u\in L^\infty \left( [0,\tau]; {\mathbb R}^n \right)$, and $r>0$ such that $\overline{B_r(x^{(0)})}\subset X_D$,
assume that
\begin{equation}\label{h}
h := \rho_0(x^{(0)},u) \rho_1^2 \rho_2 \le \frac{1}{2},
\end{equation}
\begin{equation}\label{r0}
r_0:= \frac{1-\sqrt{1-2h}}{h}\rho_0(x^{(0)},u) \rho_1 \le r,
\end{equation}
where $\rho_i$ are defined in~\eqref{rho0},~\eqref{rho1}, and~\eqref{rho2}.

Then, there exists an $x^*\in \overline{B_{r_0}(x^{(0)})}$ such that $P(x^*)=0$, and both the modified~\eqref{newton_modified} and classical~\eqref{newton_iterations} schemes converge to $x^*$ as $k\to +\infty$. The convergence rates of these schemes are estimated by the following inequalities:
\begin{equation}\label{conv_classical}
\|\tilde x^{(k)}-x^*\| \le 2^{1-k} (2h)^{2^k-1} \eta,\quad k=0,1,...\, ,
\end{equation}
\begin{equation}\label{conv_modified}
\|x^{(k)}-x^*\| \le \frac{\eta}{h}\left(1-\sqrt{1-2h}\right)^{k+1},\quad k=0,1,...\, ,
\end{equation}
where $\eta = \rho_0(x^{(0)},u) \rho_1$.
Moreover, if $h < \frac12$, then the zero $x^*$ of $P$ is unique in $B_r(x^{(0)})\cap B_{r_1}(x^{(0)})$,  where $r_1 =\frac{1+\sqrt{1-2h}}{h}\eta$;  and if $h=\frac12$, then $x^*$ is unique in $\overline{B_{r_0}(x^{(0)})}$.}
\end{theorem}
The proof of this result is based on the Newton--Kantorovich theorem~\cite{kantorovich1948newton,kantorovich1982functional,fernandez2020mild} and is presented in the Appendix.

\section{A controlled chemical reaction model}\label{application}

In this section, we apply the proposed iterative schemes to study periodic trajectories of the non-isothermal chemical reaction model of order $\gamma$ considered in~\cite{zuyev2017isoperimetric}:
\begin{equation}\label{sys_chem}
\small
\dot x(t) = Ax(t) + g(x(t)) + u(t),\;\; x(t)\in D,\, u(t)\in U,\,t\in [0,\tau],
\end{equation}
where
$$
\begin{aligned}
&D= \{(x_1,x_2)^\top \in {\mathbb R}^2 \,\vert\, x_1>-1,\,x_2>-1\},\\
&U= \{(u_1,u_2)^\top \in {\mathbb R}^2 \,\vert\, u_1\in [u_1^{min},u_1^{max}],\,u_2\in [u_2^{min},u_2^{max}]\},\\
&A = \begin{pmatrix}-\phi_1 & 0\\
0 & -\phi_2
\end{pmatrix},\;
g(x) = \begin{pmatrix}k_1 e^{-\varkappa} - (x_1+1)^\gamma e^{-\varkappa/(x_2+1)}\\
k_2 e^{-\varkappa} - (x_1+1)^\gamma e^{-\varkappa/(x_2+1)}\end{pmatrix},
\end{aligned}
$$
and $k_1$, $k_2$, $\gamma$, $\varkappa$, $\phi_1$, $\phi_2$ are real parameters.
System~\eqref{sys_chem} describes the deviations of the dimensionless concentration $x_1(t)$ and temperature $x_2(t)$ from their reference values in a reactor, where the inlet concentration and temperature are controlled by $u(t)$.
This system admits the trivial equilibrium $x_1=x_2=0$ for $u(t)\equiv 0$, which corresponds to steady-state reactor operation.
In the general case, the function $u(t)$ can encode complex control scenarios and is assumed to be of class $L^\infty\bigl( [0,\tau]; U \bigr)$.
The constraints $x_1>-1$ and $x_2>-1$ in $D$ postulate that the corresponding physical concentration and absolute temperature in Kelvin must be positive.
The physical meaning of the parameters of system~\eqref{sys_chem} is discussed in~\cite{zuyev2017isoperimetric}, and we take the following values for numerical simulations~\cite{benner2019periodic}:
\begin{equation}\label{parameters}
\small
\begin{aligned}
&\gamma=\phi_1=\phi_2=1,\;k_1=5.819\cdot 10^7,\;k_2=-8.99\cdot 10^5,\\
&\varkappa=17.77,\, u_1^{max}=-u_1^{min} = 1.798,\, u_2^{max}=-u_2^{min}=0.06663.
\end{aligned}
\end{equation}
This choice of parameters corresponds to the hydrolysis reaction
$\mathrm{(CH_3CO)_2 O + H_2 O \to 2\, CH_3 COOH}$ considered in~\cite{zuyev2017isoperimetric,benner2019periodic}, where
$x_1(t)$ represents the dimensionless deviation of the product concentration from its steady-state value. The control constraints in~\eqref{sys_chem},~\eqref{parameters} allow for the modulation of the inlet reactant concentration $u_1(t)$ with a physical amplitude of $85\,\%$, and the variation of the dimensionless temperature $u_2(t)\in [u_2^{min},u_2^{max}]$ corresponds to controlling the temperature in the inlet stream from 275~K to 315~K around the reference temperature of 295~K.

For the periodic problem of maximizing the mean product concentration in this realistic reaction, it was shown in~\cite{zuyev2017isoperimetric} that each optimal solution is achieved through a class of bang-bang controls. However, the question of the existence of
$\tau$-periodic solutions for system~\eqref{sys_chem} under a given bang-bang control remains unresolved. We apply the iterative schemes developed in Sections~\ref{iterative} and~\ref{newton} to compute
$\tau$-periodic trajectories for the considered model with an arbitrary admissible control over the interval $[0,\tau]$.

To test the simple iteration method described in Section~\ref{iterative}, we fix a grid size $n_G\in \mathbb N$ and consider the uniform partition of~$[0,\tau]$ with the step size~$\Delta t =\tau/n_G$:
$t_j = j\Delta t$ for $j=0, 1, ..., n_G$.
Let $x^{(k)}_j$ denote the value of an approximate solution of the operator equation ${\cal F} x(\cdot) = x(\cdot)$ at $t=t_j$, corresponding to the $k$-th iteration from Theorem~\ref{convtheorem}, {with the periodic boundary conditions~\eqref{affine_BC},~\eqref{BC_periodic_abstract}}.
We start from the trivial initial approximation $x^{(0)}_j=0$ for $j=\overline{0,n_G}$.
Then, the iteration with index $k$ is obtained from $\{x^{(k-1)}_j\}_{j=0}^{n_G}$ by applying
the rectangle quadrature rule to approximate the integrals in~\eqref{F_op} and~{\eqref{c_fun_old}}.
The resulting simple iteration method, {with a specified number of iterations  $n_I$, for nonlinear control systems of the form~\eqref{sys_switchings} is referred to as Algorithm~\ref{alg}.}

\begin{algorithm}
\caption{Simple iteration method}\label{alg}
\begin{algorithmic}
\Require  $A$, $g(\xi)$, $u(t)$, $\tau$, $n_G$, $n_I$
\Ensure $x^{(n_I)}_j\approx x(t_j)$ is a discrete-time approximation of the $\tau$-periodic solution $x(t)$ of system~\eqref{sys_switchings} after $n_I$ iterations
\State $\Delta t \gets \tau/n_G$ \Comment{Time step size}
\For{$j=0$ to $n_G$}
\State $t_j \gets j * \Delta t$
\State $x_j^{(0)} \gets 0_n$ \Comment{$0_n$ is the $n$-dimensional column vector of zeros}
\State $E_j \gets e^{t_j A}$  \Comment{Matrix exponentials}
\State $E_{-j} \gets e^{-t_j A}$
\EndFor
\State $M_0 \gets (E_{-n_G}-I)^{-1}$ \Comment{The inverse of $e^{-\tau A}-I$}
\For{$k=1$ to $n_I$} \Comment{$k$ is the iteration number}
\State $S_0 \gets 0_n$
    \For{$j=1$ to $n_G$} \Comment{$S_j \approx \int_0^{t_j} e^{-sA}(u(s)+g(x^{(k-1)}(s))) ds$}
        \State $S_j \gets S_{j-1} + \Delta t * E_{1-j}*(u(t_{j-1})+g(x^{(k-1)}_{j-1}))$
    \EndFor
    \State $c \gets M_0 * S_{n_G}$ \Comment{The approximation of $c$ in~\eqref{c_fun}}
    \For{$j=0$ to $n_G$} \Comment{The approximation of $x^{(k)}={\cal F} x^{(k-1)}$ at $t=t_j$}
        \State $x_j^{(k)} \gets E_j*(c+S_j)$
    \EndFor
\EndFor
\end{algorithmic}
\end{algorithm}

The modified Newton's method, defined by the sequence~\eqref{newton_modified}, is implemented in Algorithm~\ref{alg2}. {In both algorithms}, we choose the initial approximation $x^{(0)}(t)\equiv 0$. In this case, the variational system of equations~\eqref{Phi_x} becomes autonomous, and the corresponding fundamental matrix $\Phi_{x^{(0)}}(t)$ is efficiently represented by the matrix exponential:
$
\Phi_{x^{(0)}}(t) = e^{t(A + g'(0))}$ for all $t\in{\mathbb R}$.
The above representation allows us to evaluate the linear operator $[P'(x^{(0)})]^{-1}$ using the construction from Lemma~\ref{lemma_Pinv}.

\begin{algorithm}
\caption{Modified Newton's method}\label{alg2}
\begin{algorithmic}
\Require  $A$, $g(\xi)$, $u(t)$, $\tau$, $n_G$, $n_I$
\Ensure $x^{(n_I)}_j\approx x(t_j)$ is the $n_I$-th iteration of the modified Newton's method~\eqref{newton_modified}
\State $\Delta t \gets \tau/n_G$, $t_0 \gets 0$ \Comment{Time step and initial time}
\State $x_0^{(0)} \gets 0_n$ \Comment{Initial approximation: $0_n$ ($n$-dimensional column of zeros)}
\State $G_0 \gets g'(0)$ \Comment{The Jacobian matrix of $g(\xi)$ at zero}
\State $M_x \gets 0_{n\times n}$ \Comment{Zero matrix of size $n\times n$}
\State $E_0 \gets I$, $\Psi_0 \gets I$, $\varphi_0\gets G_0$, $Eg_0 \gets G_0$\Comment{$I$ is the identity matrix of size $n\times n$}
\For{$j=1$ to $n_G$}
\State $t_j \gets j * \Delta t$, $x_j^{(0)} \gets 0_n$ \Comment{Time discretization and the initial approximation}
\State $E_j \gets e^{t_j A}$, $E_{-j} \gets e^{-t_j A}$  \Comment{Matrix exponentials}
\State $\Psi_j \gets e^{t_j (A+G_0)}$, $\varphi_j \gets \Psi_j * G_0$, $Eg_j \gets E_j * G_0$
\State $\Psi_{-j} \gets e^{-t_j (A+G_0)}$, $\varphi_{-j} \gets \Psi_{-j} * G_0$, $Eg_{-j} \gets E_{-j} * G_0$
\State $M_x \gets M_x + Eg_{1-j} * \Psi_{j-1}$
\EndFor
\State $M_0 \gets (E_{-n_G}-I)^{-1}$ \Comment{The inverse of $e^{-\tau A}-I$}
\State $M_x \gets \Delta t* M_0 * M_x - I$, $M\gets M_x^{-1} * M_0$ \Comment{$M_x$ in~\eqref{M_x} and the product of matrices for~\eqref{C_form}}
\For{$k=1$ to $n_I$} \Comment{$k$ is the iteration number}
    \State $S_0 \gets 0_n$
    \For{$j=1$ to $n_G$} \Comment{$S_j \approx \int_0^{t_j} e^{-sA}(u(s)+g(x^{(k-1)}(s))) ds$}
        \State $S_j \gets S_{j-1} + \Delta t * E_{1-j}*(u(t_{j-1})+g(x^{(k-1)}_{j-1}))$
    \EndFor
    \State $c \gets M_0 * S_{n_G}$ \Comment{The approximation of $c$ in~\eqref{c_fun}}
    \For{$j=0$ to $n_G$} \Comment{$dy_j\approx P(x^{(k-1)} )(t_j)$ in~\eqref{p_map}}
        \State $dy_j \gets E_j*(c+S_j) - x^{(k-1)}_j$
    \EndFor
    \State $S_0 \gets 0_n$, $Cs \gets 0_n$
    \Comment{Compute the action of $[P'(x^{(0)})]^{-1}$ on $P(x^{(k-1)} )$}
    \For{$j=1$ to $n_G$}
     \State $S_j \gets S_{j-1} + \Delta t * \varphi_{1-j}* dy_{j-1}$
       \Comment{$S_j\approx \int_0^{t_j} \Phi_x^{-1}(s)g'(0)dy(s) ds$ in~\eqref{C_form}}
     \State $Cs \gets Cs+ \Delta t * Eg_{1-j}* (dy_{j-1} + \Psi_{j-1} * S_{j-1})$
    \EndFor \Comment{$Cs\approx \int_0^\tau e^{-tA} g'(0) \left\{ d y(t) + \Phi_x(t)\int_0^t \Phi_x^{-1}(s) g'(0) d y(s) ds \right\} dt$ in~\eqref{C_form}}
    \State $C_x \gets M * Cs$ \Comment{$C_x$ in~\eqref{C_form}}
    \For{$j=0$ to $n_G$}
       \State $x_j^{(k)} \gets x_j^{(k-1)} +dy_{j} - \Psi_j * (C_x - S_j)$
    \EndFor \Comment{$x_j^{(k)}$ is the discrete-time approximation of $x^{(k)}(t_j)$ in~\eqref{newton_modified}}
\EndFor
\end{algorithmic}
\end{algorithm}

Algorithms~\ref{alg} and~\ref{alg2} have been implemented in Maple~2020 with the use of the {\texttt MatrixExponential} function to evaluate $e^{t A}$.
We approximate the corresponding integrals in~\eqref{dx}--\eqref{M_x} by the rectangle quadrature rule.
In the simulations, we define the function $u(t)$ to be constant on each subinterval of the partition
$
0=\tau_0 < \tau_1 < \cdots <\tau_N = \tau
$,
i.e.,
\begin{equation}\label{u_bangbang}
u(t) = \sum_{i=1}^N u^{(i)}\chi_{[\tau_{i-1},\tau_i)}(t),\; u(\tau)=u^{(N)},
\end{equation}
where $\chi_{[\tau_{i-1},\tau_i)}(t)$ is the indicator function of $[\tau_{i-1},\tau_i)$.
The function $u(t)$ in~\eqref{u_bangbang} corresponds to a family of bang-bang controls with the values $u^{(i)}\in \partial U$ at the boundary of $U$.
\begin{figure}[thpb]
	\centering
		\includegraphics[width=0.9\linewidth]{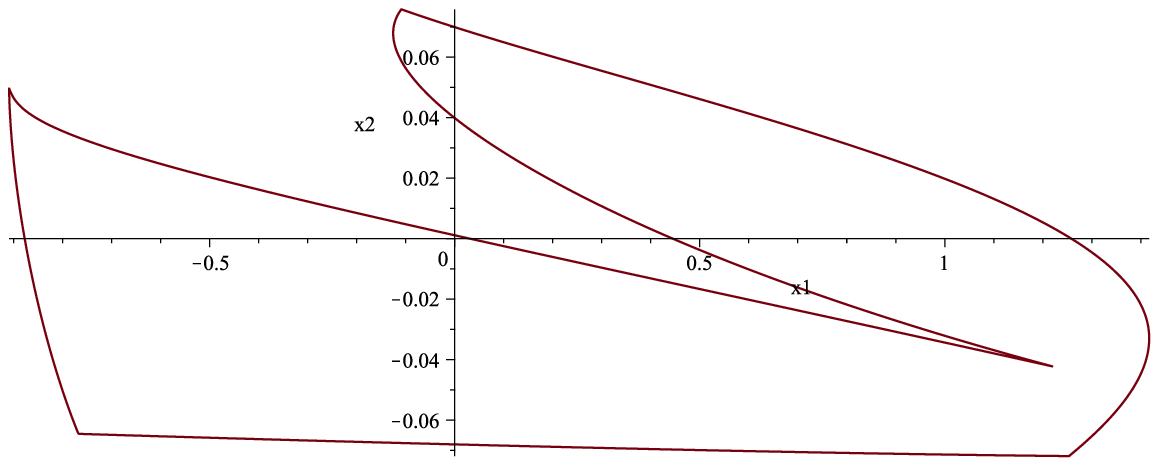}
		\caption{Periodic trajectory of~\eqref{sys_chem} with $u(t)$ of the form~\eqref{u_bangbang}--\eqref{tau_5}, $\tau=10$.}
	\label{Fig_N5_10}
\end{figure}

To illustrate the behavior of periodic solutions of system~\eqref{sys_chem} with discontinuous $u(t)$ of the form~\eqref{u_bangbang}, we fix $N=5$ and consider the switching sequence
\begin{equation}
u^{(1)} = u^{(3)}=-u^{(4)}=\begin{pmatrix}u_1^{max} \\ u_2^{min}\end{pmatrix},\;
u^{(2)}=-u^{(5)}=\begin{pmatrix}u_1^{max} \\ u_2^{max}\end{pmatrix}
\label{switching_5}
\end{equation}
together with the following switching time parameterization:
\begin{equation}
\small
\tau_0=0,\, \tau_1 = 0.1\tau,\,\tau_2 = 0.3\tau,\, \tau_3 = 0.5\tau,\, \tau_4 = 0.8\tau,\,\ \tau_5=\tau.
\label{tau_5}
\end{equation}
Algorithms~\ref{alg} and~\ref{alg2} have been executed for system~\eqref{sys_chem} with the above choice of $u(t)$ on a grid of size $n_G=10^5$ and with the number of iterations $n_I=9$.
For the simulations, we have used Maple~2020 software running on a laptop with an Intel Core i7 processor.
The results of Algorithm~\ref{alg2}
are shown in Fig.~\ref{Fig_N5_10} as the plot of  $x^{(n_I)}_j=(x^{(n_I)}_{j1},x^{(n_I)}_{j2})^\top$, $j=\overline{0,n_G}$.

We check the accuracy of Algorithms~\ref{alg} and~\ref{alg2} by computing the residuals
based on the discrete-time version of the governing integral equation~\eqref{intexpA}:
\begin{equation}\label{discr}
\small
d^{(k)} = \max_{0< j\le n_G} \left\| x^{(k)}_j - e^{t_j A}x^{(k)}_0 - \Delta t \sum_{i=0}^{j-1} e^{(t_j-t_i) A} \left( g( x^{(k)}_i ) + u(t_i) \right) \right\|,
\end{equation}
where $k$ is the iteration number.
The above $d^{(k)}$ measures the difference between the finite-difference approximation of $x(t)$
and the right-hand side of the integral equation~\eqref{intexpA} evaluated at the grid points. To check that the constructed solutions satisfy the periodic boundary conditions, we also evaluate $\|x^{(k)}(0)-x^{(k}(\tau)\|\approx \|x^{(k)}_0 - x^{(k)}_{n_G}\|$ at each iteration.
The simulation results are summarized in Table~1.

As shown in Table~1, the resulting residuals (which measure the approximation error) decrease monotonically as the iteration number $k$ increases.
The modified Newton's method (Algorithm~\ref{alg2}) ensures considerably faster convergence compared to the simple iteration method (Algorithm~\ref{alg}).
We see that all the approximate solutions are $\tau$-periodic, considering the numerical tolerance of the computations.
A significant reduction in residuals is already achieved in Algorithm~\ref{alg2} by iteration $k=8$. Since the rectangle quadrature rule is used to approximate the integrals with {\em merely measurable} input functions $u\in L^\infty[0,\tau]$, it is natural to expect that there will be no significant decay in residuals in subsequent iterations of the modified Newton's method, especially when the precision of numerical computations becomes dominant.

\begin{table}[h!]
\begin{center}
\begin{tabular}{|c|c|c|c|c|}
\hline
\multirow{2}{*}{} & \multicolumn{2}{c|}{\textbf{Algorithm 1}} & \multicolumn{2}{c|}{\textbf{Algorithm 2}}
\\
\hline
$\mathbf k$  &  $\mathbf{d^{(k)}}$ & $\mathbf{\|x^{(k)}_0 - x^{(k)}_{n_G}\|}$ &  $\mathbf{d^{(k)}}$ & $\mathbf{\|x^{(k)}_0 - x^{(k)}_{n_G}\|}$ \\ \hline
0 & 0.440438 & 0 & 0.440438& 0 \\ \hline
1 & 0.0650220 & $2.8319 \cdot 10^{-11}$ & 0.00569119& $2.1651 \cdot 10^{-15}$ \\ \hline
2 & 0.0102533 & $2.7931 \cdot 10^{-11}$ & 0.000180856& $2.2205 \cdot 10^{-15}$  \\ \hline
3 & 0.00301579 & $2.7289\cdot 10^{-11}$ & $3.22370 \cdot 10^{-6}$ &$2.1650 \cdot 10^{-15}$ \\ \hline
4 & 0.00173071 & $2.7452 \cdot 10^{-11}$ & $4.70956\cdot 10^{-8}$ & $2.1095 \cdot 10^{-15}$\\ \hline
5 & 0.00132163 & $2.7353\cdot 10^{-11}$ & $6.39264\cdot 10^{-10}$ & $2.1650 \cdot 10^{-15}$\\ \hline
6 & 0.00108846 & $2.7438 \cdot 10^{-11}$ & $6.64978\cdot 10^{-12}$ & $2.1095 \cdot 10^{-15}$\\ \hline
7 & 0.000886124 & $2.7368 \cdot 10^{-11}$ & $5.49621\cdot 10^{-14}$ &$2.1650 \cdot 10^{-15}$ \\ \hline
8 & 0.000721299 & $2.7425 \cdot 10^{-11}$ & $3.88675 \cdot 10^{-16}$ & $2.2205 \cdot 10^{-15}$ \\ \hline
9 & 0.000587331& $2.7378 \cdot 10^{-11}$ & $2.22214 \cdot 10^{-16}$ & $2.1650 \cdot 10^{-15}$ \\ \hline
\end{tabular}
\label{tab1}
\caption{Residuals $d^{(k)}$ for Algorithms~\ref{alg} and~\ref{alg2} with $n_G=10^5$ and $\tau=1$.}
\end{center}
\end{table}

\section{Conclusions and future work}\label{conc}
\vskip-1ex
The key contribution of this paper establishes existence and uniqueness conditions for solutions to nonlinear control systems of the form~\eqref{sys_switchings} with discontinuous input functions (of class $L^\infty$), general nonlinearities (of class $C^1$),
{and boundary conditions of the general affine form~\eqref{affine_BC}.}
Our crucial assumption in {Theorem~\ref{convtheorem}} is the dominant linearization condition {$\det ({\cal B}_\tau) \neq 0$, which, in the case of periodic boundary conditions, reduces to $\det \left(e^{-\tau A}-I\right)\neq 0$ and} relates the period $\tau$ to the spectral properties of the exponential of $A$.

It should be emphasized that sufficient convergence conditions
{are explicitly presented: for the simple iteration method (Theorem~\ref{convtheorem}, under general boundary conditions) and for Newton's method (Theorem~\ref{thm_NK}, under periodic boundary conditions).}
In Theorem~\ref{thm_NK}, the theoretical convergence rates are {analytically estimated} by formulas~\eqref{conv_classical} and~\eqref{conv_modified}.
Our case study justifies the applicability of the proposed iteration schemes to nonlinear control systems describing periodic non-isothermal chemical reactions {with arbitrary discontinuous inputs.
Due to the considerable practical importance of such reactions
in chemical engineering (as outlined in~\cite{felischak2021}), the proposed algorithms  hold substantial promise for further applications.
To our knowledge, no comparable concept based on a rigorous analytical approach
exists in the literature.}
%
%
The presented algorithms allow the use of general discontinuous functions $u(t)$ without any regularity assumptions except for integrability on $[0,\tau]$.
Because of this generality, no high-order quadrature formulas have been incorporated into Algorithms~1 and~2.
As a potential direction for improvement, adaptive quadrature rules can be applied to approximate integrals involving $u(t)$, taking into account the regularity of input signals on the intervals $(\tau_i,\tau_{i+1})$ within the context of the bang-bang parameterizations considered in Section~\ref{application}.
We consider the development of such adaptive schemes,
as well as the numerical implementation of Newton's iterations in the form of~\eqref{newton_iterations}
with an efficient inversion of $P'$, to be topics for future research.

 We also note that for problems with a higher number of state variables, the efficiency of the algorithms can be improved by modern techniques from numerical linear algebra that compute the action of matrix exponentials to vectors without explicitly forming the notoriously problematic matrix exponential explicitly. This will be explored when studying the application of our approach for more complex reaction networks.

\appendix
\section{Appendix}


{\em Proof of Lemma~\ref{lemma_Y}.}
We denote the matrix
$
B = \left(e^{-\tau A}-I\right)^{-1}\int_0^\tau e^{-tA} \left.\frac{\partial g(x)}{\partial x}\right\vert_{x(t)} \Phi_x(t) dt
$,
and use the Neumann series for the inverse of $M_x$ defined in~\eqref{M_x}:
\begin{equation}\label{M_x1}
M_x^{-1} = -(I-B)^{-1} = -\sum_{k=0}^\infty B^k.
\end{equation}
The above series converges under assumption $(A6)$:
$
\|B\|\le R_\tau ML \int_0^\tau e^{\omega t}\phi_L(t) dt = S<1
$,
where the constants $M$, $L$, $\omega$, $R_\tau$ are defined in $(A2)$ and~\eqref{RT},
and $\|\Phi_x(t)\|\le \phi_L(t) =\sqrt{n}e^{(\|A\|+L)t}$ is estimated by Lemma~\ref{lemma_X}.
Moreover,~\eqref{M_x1} together with the triangle inequality implies the assertion of Lemma~\ref{lemma_Y}:
$$
\|M_x^{-1}\| \le \sum_{k=0}^\infty \|B^k\| \le \sum_{k=0}^\infty S^k = \frac{1}{1-S}.
$$
$\square$

{\em Proof of Theorem~\ref{thm_NK}.}
The nonlinear operator $P:X_D\to X$ defined by~\eqref{p_map} is Fr\'echet differentiable at each $x\in X_D$,
and $\Gamma_0:=[P'(x^{(0)})]^{-1}\in {\mathcal L}(X,X)$ with
\begin{equation}\label{Gamma0}
\|\Gamma_0\|\le \rho_1
\end{equation}
due to Lemma~\ref{lemma_Pinv}.
We estimate $\|P(x)\|$ by applying the triangle inequality and H\"older's inequality to~\eqref{p_map},~\eqref{F_op},~and~{\eqref{c_fun_old}} under assumption~$(A2)$:
$$
\|P(x)\| \le \|x\| + Me^{\omega\tau} \|c(x)\| + \esssup_{s\in [0,\tau]}\|g(x(s))+u(s)\| \int_0^\tau \|e^{(\tau-s)A}\|ds
$$
and
$$
\|c(x)\| \le R_\tau \esssup_{s\in [0,\tau]}\|g(x(s))+u(s)\| \int_0^\tau \|e^{-s A}\|ds,
$$
where $R_\tau$ is given by~\eqref{RT}. The above inequalities and~$(A2)$ imply that
\begin{equation}\label{P0}
\|P(x)\| \le \rho_0(x,u)\quad \text{for all}\; x\in X_D,\;u\in L^\infty \left([0,\tau];{\mathbb R}^n\right),
\end{equation}
provided that $\rho_0(x,u)$ is defined in~\eqref{rho0}.

It remains to prove that the second derivative of $P(x)$ exists and to estimate the norm of $P''(x)$.
We recall that $P''(x):X\times X\to X$ at $x\in X_D$ is a bilinear operator~\cite[Ch.~XVII]{kantorovich1982functional}:
\begin{equation}\label{P2}
P''(x)(x^1,x^2) = \lim_{\varepsilon\to 0} \frac{1}{\varepsilon}\Bigl( P'(x+\varepsilon x^2)x^1 - P'(x)x^1\Bigr),
\end{equation}
provided that the above limit exists uniformly with respect to $x^1,x^2\in X$ with $\|x^1\|=\|x^2\|=1$.
As $P(x)={\cal F}(x)-x$, then $P''(x)={\cal F}''(x)$ if the latter second derivative exists.
Given $x\in X_D$, $x^1,x^2\in X$, the value of $P''(x)(x^1,x^2)$ at $t\in [0,\tau]$ is thus defined
by exploiting~\eqref{dF},~{\eqref{dc_old}}, and~\eqref{P2} as follows:
\begin{equation}\label{P2t}
\begin{aligned}
& P''(x)(x^1,x^2)(t) = \lim_{\varepsilon\to 0} \frac{1}{\varepsilon}\Bigl( e^{tA} \left( c'_{x+\varepsilon x^2} (x^1)- c'_{x} (x^1) \right) \\
&+\int_0^t e^{(t-s)A} \left[ g'(x(s)+\varepsilon x^2(s)) - g'(x(s)) \right] x^1(s) ds\Bigr).
\end{aligned}
\end{equation}
Since $g\in C^2(D;{\mathbb R}^2)$, the second derivative of $g$ at any $\xi \in D$ is a bounded bilinear operator $g''(\xi):{\mathbb R}^n\times {\mathbb R}^n \to \mathbb R^n$, and the limit in~\eqref{P2t} is well-defined. The components of the column vector $g''(\xi)(\xi^1,\xi^2)$ for given $\xi^1=(\xi^1_1,...,\xi^1_n)^\top$ and
$\xi^2=(\xi^2_1,...,\xi^2_n)^\top$ are:
$
g''(\xi)(\xi^1,\xi^2)_k = \sum_{i,j=1}^n \frac{\partial^2 g_k(\xi)}{\partial \xi_i \partial \xi_j} \xi^1_i \xi^2_j$, $k=1,2,...,\, n$.
Assumption~$(A6)$, together with the Cauchy--Schwarz inequality, implies that
\begin{equation}\label{g2_norm}
\|g''(\xi)(\xi^1,\xi^2)\| \le \sqrt{n} \bar H \|\xi^1\|\cdot\|\xi_2\|\quad \text{for all}\; \xi^1,\xi^2\in {\mathbb R}^n.
\end{equation}
Then, the limit in~\eqref{P2t} takes the following form with $g''$:
\begin{equation}\label{P2g2}
\begin{aligned}
&P''(x)(x^1,x^2)(t) =  \int_0^t e^{(t-s)A} g''(x(s)) (x^1(s),x^2(s)) \, ds \\
&+e^{tA} (e^{-\tau A}-I)^{-1}\int_0^\tau e^{-sA} g''(x(s)) (x^1(s),x^2(s)) \, ds.
\end{aligned}
\end{equation}
Formula~\eqref{P2g2} defines the bounded bilinear operator $P''(x):X\times X \to X$ for any $x\in X_D$.
Indeed, to show the boundedness of $P''(x)$, we estimate the norm of $P''(x)(x^1,x^2)$ in~\eqref{P2g2} using the triangle inequality, H\"older's inequality, and estimate~\eqref{g2_norm}:
$$
\begin{aligned}
\|P''(x)(x^1,x^2)\| &\le \sup_{s\in [0,\tau]} \|g''(x(s)) (x^1(s),x^2(s))\|
( M R_\tau e^{\omega \tau} \int_0^\tau \|e^{-sA}\|ds+ \int_0^\tau \|e^{(\tau-s)A}\|ds ) \\
&\le \rho_2 \|x^1\|\cdot \|x^2\|\;\;\text{for all}\;\;x^1,x^2\in X_D,
\end{aligned}
$$
where $\rho_2$ is defined by~\eqref{rho2}. The above inequality means that
\begin{equation}\label{P2norm}
\|P''(x)\| \le \rho_2\quad\text{for all}\; x\in X_D.
\end{equation}

Given $x^{(0)}\in X_D$ and $u\in L^\infty \left( [0,\tau]; {\mathbb R}^n \right)$, conditions~\eqref{Gamma0},~\eqref{P0},~\eqref{P2norm},~\eqref{h}, and~\eqref{r0} imply that the assumptions of the Newton--Kantorovich theorem~\cite[Ch.~XVIII, Theorem~6]{kantorovich1982functional} (see also \cite[Theorem~1.1]{fernandez2020mild}) are satisfied with:
$
\|\Gamma_0 P(x^{(0)})\|\le \eta = \rho_1 \rho_0(x^{(0)},u)$, $\|\Gamma_0 P''(x))\|\le K = \rho_1 \rho_2$ {for all} $x\in D$, and
$
h = K\eta = \rho_0(x^{(0)},u) \rho_1^2 \rho_2 \le \frac12
$.
Then, the assertions of Theorem~\ref{thm_NK} follow from the Newton--Kantorovich theorem.
$\square$

\end{document}